\documentclass[12pt]{amsart}
\usepackage[english]{babel}

\usepackage{pifont}
\usepackage{todonotes}

\usepackage{geometry}
\usepackage{amsthm,amsmath,amssymb,amsfonts,amscd,mathtools}
\usepackage[shortlabels]{enumitem}
\usepackage{hyperref,cleveref}
\usepackage{multirow,float}

\theoremstyle{plain}

\newtheorem{theorem}{Theorem}[section]

\newtheorem{lemma}[theorem]{Lemma}
\newtheorem{proposition}[theorem]{Proposition}

\newtheorem{corollary}[theorem]{Corollary}

\theoremstyle{definition}
\newtheorem{definition}[theorem]{Definition}
\newtheorem{remark}[theorem]{Remark}

\theoremstyle{remark}
\newtheorem{example}[theorem]{Example}

\numberwithin{equation}{section}
\numberwithin{figure}{section}
\numberwithin{table}{figure}

\newcommand{\pt}[1]{\left({#1}\right)}
\newcommand{\pq}[1]{\left[{#1}\right]}

\newcommand{\rest}[2]{\left.{#1}\right|_{#2}}
\newcommand{\pg}[1]{\left\{{#1}\right\}}
\newcommand{\abs}[1]{\left|{#1}\right|}

\newcommand{\vol}{\operatorname{Vol}}

\newcommand{\R}{\mathbb{R}}

\newcommand{\C}{\mathbb{C}}

\newcommand{\ep}[2]{e_{\alpha_{#1}^{#2}}}
\newcommand{\emp}[2]{e_{-\alpha_{#1}^{#2}}}
\newcommand{\adws}[1]{\pq{W,\sigma\pt{#1}}}

\DeclareMathOperator{\SU}{SU}

\DeclareMathOperator{\I}{I_\lambda}

\title[]{$p$-K\"ahler structures on fibrations and reductive Lie groups}

\author{Anna Fino}
\address[Anna Fino]{Dipartimento di Matematica \lq\lq Giuseppe Peano\rq\rq \\ Universit\`a di Torino\\
Via Carlo Alberto 10\\
10123 Torino\\ Italy\\
and  Department of Mathematics and Statistics Florida International University\\
  Miami Florida, 33199, USA}
 \email{annamaria.fino@unito.it, afino@fiu.edu}

\author{Gueo Grantcharov}
\address[Gueo Grantcharov]{Department of Mathematics and Statistics \\
Florida International University\\
Miami, FL 33199, United States
\& Institute of Mathematics and
Informatics, Bulgarian Academy of Sciences\\
 8 Acad. Georgi Bonchev Str. 1113 Sofia, Bulgaria}
\email{grantchg@fiu.edu}

\author{Asia Mainenti}
\address[Asia Mainenti]{“Simion Stoilow” Institute of Mathematics of the Romanian Academy\\
 Calea Grivi\c{t}ei 21, Bucharest, Romania}
 \email{asia.mainenti@imar.ro}
 
\thanks{}
\keywords{quasi-regular  fibrations,  reductive Lie groups,  $p$-Kähler  structures, balanced metrics}
\subjclass[2020]{Primary: 53C55; Secondary: 53C25, 22E46}

\begin{document}

\begin{abstract}  
We investigate the existence of  $p$-K\"ahler structures  on  two classes of complex manifolds: on quasi-regular  fibrations, with particular emphasis on complex homogeneous spaces, and on reductive Lie groups endowed with invariant complex structures. In the latter setting, we construct  non-regular complex structures on the Lie algebras $\mathfrak{sl}(2m-1,\mathbb{R})$ for $m \ge 2$ and show that these structures admit compatible balanced metrics, providing   new  explicit examples of balanced manifolds.
\end{abstract}

 \maketitle

\section{Introduction}

$p$-K\"ahler structures are generalizations of K\"ahler metrics, defined as closed transverse $(p,p)$-forms \cite{AA}. Transversality is a positivity notion that we will better discuss in what follows. For $p=1$, we recover K\"ahler metrics, whereas $(n-1)$-K\"ahler structures are $(n-1)$-th powers of balanced Hermitian metrics. The case $p=n$ is trivial, as every top form is closed. Manifolds admitting such structures were extensively studied by L. Alessandrini and G. Bassanelli, especially on compact holomorphically parallelizable manifolds, smooth proper modifications of compact K\"ahler manifolds, and in comparison with the K\"ahler and balanced case. More recently, their behavior under small deformations of the complex structure has been studied, see for instance \cite{LG} and the references therein.

Certain quasi-regular  fibrations, including examples arising from compact complex homogeneous spaces, can be endowed with balanced metrics, providing natural instances of non-K\"ahler manifolds supporting special Hermitian structures.
However, very little is known about the existence of $p$-K\"ahler structures on these manifolds. In \cite{FM24,FMarxiv}, the first and third authors obtained obstruction results for the existence of $p$-K\"ahler structures in the class of solvmanifolds, in particular on nilmanifolds and on solvmanifolds whose associated Lie algebras contain large abelian ideals. These results highlight the strong influence of the underlying algebraic structure on the existence of $p$-K\"ahler structures. 
In contrast,{there is minimal knowledge} about their existence on reductive Lie groups, where the presence of a semisimple component introduces new challenges and phenomena not present in the solvable setting.

In this paper, we study the existence of $p$-K\"ahler structures on quasi-regular  fibrations and reductive Lie groups endowed with invariant complex structures. Quasi-regular  fibrations form a natural class of holomorphic fibrations over complex orbifolds, encompassing many important examples, including compact complex homogeneous spaces. A particularly notable subclass of these fibrations is provided by the Tits fibration. For a compact homogeneous space $M = G/H$, the Tits fibration expresses $M$ as a holomorphic torus bundle 
\[
\pi: G/H \longrightarrow G/K,
\] 
where $K$ is the centralizer of a torus in $G$ and the fiber $H/K$ is a complex torus. This construction allows one to study the existence of $p$-K\"ahler structures on $M$ via the characteristic classes of the base $G/K$.

In Section \ref{section3} we provide obstructions to the existence of $p$-K\"ahler structures on quasi-regular  fibrations, in terms of characteristic classes of orbifold line bundles on the base. Specifically, given a compact complex orbifold $X$ and a quasi-regular  fibration $\pi: M \to X$ with an orbifold line bundle $L$ satisfying $c_1(L) \ge 0$ and $\pi^* c_1(L) = 0$, we show that $M$ does not admit an $(m-j)$-K\"ahler structure for any $1 \le j \le k$, where $k$ is determined by the non-vanishing powers of $c_1(L)$ and $m = \dim_{\mathbb{C}} M$ (\Cref{fibration}).

In the non-compact setting, analogous fibrations exist when $G$ is a reductive Lie group. We recall that a Lie group is reductive if its Lie algebra is a direct sum of an abelian ideal and a semisimple ideal. By \cite{Morimoto}, every even-dimensional, real, reductive Lie group admits invariant complex structures. Every invariant complex structure on $G_0$ is determined by a choice of an $n$-dimensional complex subalgebra $\mathfrak{q}$ of the complexification $\mathfrak{g}$, with $\mathfrak{g} = \mathfrak{q} \oplus \sigma(\mathfrak{q})$, where {$2n$ is the real dimension of $G_0$ and} $\sigma$ is the complex conjugation with respect to $\mathfrak{g}_0$. Following the classification of Snow, reductive Lie algebras split into two classes: Class I, where all simple factors of $\mathfrak{g}$ are of inner type, and Class II, consisting of all other cases.

In many cases, by \cite{Snow}, a reductive Lie group $G_0$ with an invariant complex structure can be fibered holomorphically over an open $G_0$-orbit in a homogeneous projective rational manifold. This gives an effective method to study the geometry of the complex structure in $G_0$, using the results by Wolf in \cite{Wolf}. When $G_0$ does not have such a fibration, then $G_0$ is either abelian or a complex Lie group with certain exceptional structures, but also in these cases the complex geometry of $G_0$ is understood. In the special case when there is a maximal abelian subgroup of $G_0$ which is compact and whose subalgebra is {invariant with respect to the complex structure}, the fibers are complex abelian varieties and the base is an open subset of a flag manifold. This is a dual picture to the classical Tits fibration of an even-dimensional compact Lie group over its flag manifold. From a differential geometric viewpoint, both have the structure of a holomorphic torus fibration over a complex base.

In Section \ref{section4} we prove several results concerning the existence of $p$-K\"ahler and related Hermitian structures on reductive Lie groups. In particular, for compact, non-abelian, even-dimensional reductive Lie groups $G_0$, we show that $(n-k)$-K\"ahler structures cannot exist for $1 \le k \le r_0$, where $r_0$ is the number of positive roots of the complexified Lie algebra. 
{We note that this is consistent with the analogous results in \Cref{section3} for compact homogeneous spaces, as the base of the Tits fibration in this case is a K\"ahler-Einstein manifold of complex dimension $r_0$.}
Similarly, for non-compact reductive Lie groups with regular invariant complex structures, we exhibit obstructions to the existence of $(n-2)$-pluriclosed structures, generalizing earlier results known for simple groups of inner type.

Finally,  in Section \ref{section5} we construct a family of non-regular complex structures on $\mathfrak {sl}(2m-1,\R)$ for all $m \geq 2$ and show that, while these structures admit compatible balanced metrics, they do not support pluriclosed metrics.
Related topics in the more general setting of non-compact semisimple groups are being investigated in \cite{Kwong}.

We briefly outline the structure of the paper. In Section \ref{section2}, we recall basic notions on $p$-K\"ahler and $p$-pluriclosed structures on complex manifolds,  with emphasis on invariant structures on Lie groups. Section \ref{section3} is devoted to quasi-regular fibrations, where we discuss the existence of balanced metrics and the lack of $p$-K\"ahler structures in this setting. In Section \ref{section4}, we study complex structures on even-dimensional reductive Lie groups, proving obstructions to the existence of $p$-K\"ahler and $p$-pluriclosed metrics in both compact and non-compact cases. Section \ref{section5} focuses on the case of $\mathfrak{sl}(2m - 1,\R)$.

\section{\texorpdfstring{$p$-K\"ahler and $p$}{p-Kähler and p}-pluriclosed structures} \label{section2}
In this section we recall the definition and first properties of $p$-K\"ahler and $p$-pluriclosed structures on  complex manifolds, including their relation to various types of special Hermitian metrics.
Such structures are defined combining closure, with respect to some differential operator, with positivity.
As far as the latter is concerned, there are multiple possible definitions of positivity on differential forms on a complex manifold of complex dimension $n$, all depending on the choice of a volume form $\vol$, i.e. a never vanishing $(n,n)$-form.
The definition for top forms is intuitive, as positive $(n,n)$-forms are defined as non-negative multiples of $\vol$, and strictly positive $(n,n)$-forms are positive multiples of $\vol$.
On the other hand, for $(p,p)$-forms, with $p<n$, there are many positivity notions one can define, see for instance \cite{hk}, and the distinction among some of them relies on the notion of decomposability. 
A (real) $(p,p)$-form is called \textit{decomposable} if it can be written as  $i\,\eta_1\wedge\overline{\eta_1}\dots\wedge\,i\,\eta_p\wedge\overline{\eta_p}$, for some $(1,0)$-forms $\eta_1,\dots,\eta_p$. 
\begin{definition}
A \textit{strongly positive} $(p,p)$-form is a convex combination of {decomposable} $(p,p)$-forms.
\end{definition}
This defines a closed convex cone in the space of $(p,p)$-forms, and via the duality pairing given by the wedge product, we can define the dual cone, whose elements we will call \textit{weakly positive} forms. 
The forms in the interior of both the cones are called positive in the strict sense and they are non-degenerate, in the sense that their restrictions to any $p$-dimensional complex sub-bundle of the tangent bundle, are volume forms. 
For the sake of clarity, we write down explicitly the definition of strictly weakly positive forms.
\begin{definition}
    A $(p,p)$-form $\Omega$ is \textit{strictly weakly positive}, or \textit{transverse} if and only if $\Omega\wedge\Theta$ is a positive multiple of $\vol$, for every strongly positive $(n-p,n-p)$-form $\Theta$.
\end{definition}
Despite the common terminology in use in the field being \textit{strict weak positivity}, in what follows, for simplicity, we will call such forms transverse, with the notation introduced in \cite{sullivan}.

In light of this, we can now introduce $p$-K\"ahler and $p$-pluriclosed structures on a complex manifold.

\begin{definition}[\cite{AA,Ales}]
    Let $(M,J)$ be a complex manifold of complex dimension $n$, let $p$ be an integer with $0\le p\le n$, and $\Omega$ a (real) transverse $(p,p)$-form.
Then, $\Omega$ is \textit{$p$-K\"ahler} if $d\Omega=0$, and  $\Omega$ is \textit{$p$-pluriclosed} if $\partial\bar\partial\Omega=0$.
The complex manifold $(M,J)$ is \textit{$p$-K\"ahler} if it admits a $p$-K\"ahler structure, and \textit{$p$-pluriclosed} if it admits a $p$-pluriclosed structure.
\end{definition}

Whilst the definition is well-posed for $0\le p\le n$, the interesting and non-trivial cases are for $1\le p\le n-1$.
For some particular values of $p$, we recover well-known notions of special Hermitian metrics, of which we recall the definitions.

\begin{definition}
Let $\omega$ be the fundamental $(1,1)$-form associated to a Hermitian metric on a complex manifold $(M,J)$ of complex dimension $n$.
    \begin{enumerate}
        \item $\omega$ is \textit{K\"ahler} if $d\omega=0$;
        \item $\omega$ is \textit{pluriclosed}, or \textit{strong K\"ahler with torsion (SKT)}, if $\partial\bar\partial\omega=0$;
        \item $\omega$ is \textit{balanced} if $d\omega^{n-1}=0$;
        \item $\omega$ is \textit{Gauduchon} if $\partial\bar\partial\omega^{n-1}=0$;
        \item $\omega$ is \textit{astheno-K\"ahler} if $\partial\bar\partial\omega^{n-2}=0$.
    \end{enumerate}
\end{definition}

\begin{remark}
Then, the following implications follow trivially from the definitions.
Let $(M,J)$ be a complex manifold of complex dimension $n$.
    \begin{enumerate}
        \item $(M,J)$ is $1$-K\"ahler if and only if it is K\"ahler;
        \item $(M,J)$ is $1$-pluriclosed if and only if it is pluriclosed;
        \item $(M,J)$ is $(n-1)$-K\"ahler if and only if it is balanced;
        \item $(M,J)$ is $(n-1)$-pluriclosed if and only if it is Gauduchon;
        \item if $(M,J)$ is astheno-K\"ahler, then it is $(n-2)$-pluriclosed;
        \item for every $0\le p\le n$, if $(M,J)$ is $p$-K\"ahler, then it is $p$-pluriclosed.
    \end{enumerate}
\end{remark}
    
    \smallskip
We conclude this section mentioning that the Harvey-Lawson criterion for the existence of K\"ahler metrics can be generalized to $p$-K\"ahler and $p$-pluriclosed structures.
\begin{proposition}[\cite{AA,Ales}]
    A compact complex manifold $(M,J)$ is $p$-K\"ahler if and only if all strongly positive currents of bidimension $(p,p)$ that are $(p,p)$-components of a boundary are trivial.
    Similarly, $(M,J)$ is $p$-pluriclosed if and only if all strongly positive currents of bidimension $(p,p)$ that are $(p,p)$-components of a $\partial\bar\partial$-boundary are trivial.
\end{proposition}
As a corollary, the existence of strongly positive, exact $(n-p,n-p)$-forms is an obstruction to the existence of $p$-K\"ahler structures, on a given compact complex manifold, and the analogous statement holds for $p$-pluriclosed structures.
A more simple proof of this result relies on the Stokes theorem, and it will be recalled in what follows, when needed.
Furthermore, we note that the same argument can be adapted to unimodular Lie algebras.

\section{Quasi-regular fibrations} \label{section3}

We recall that a  complex orbifold $X$  is  a normal complex space locally given by charts written as quotients of smooth coordinate charts. More precisely,  $X$  is a topological space together with a cover of coordinate charts where each element of the cover is homeomorphic to a quotient of an open in $\mathbb{C}^n$ containing the origin by a finite group $G_x$, $x\in X$, and the transition functions are covered by holomorphic maps.  The complex dimension of the orbifold  is the number $n$ in $\mathbb{C}^n/G_x$.

Assume that $M$ is a {compact}  complex manifold and $\pi: M\rightarrow X$ is a holomorphic fibration over  a compact complex orbifold $X$.
In the paper we will call $\pi$ a quasi-regular fibration on $M$ with a base $X$.
A special class of these fibrations is that of toric Seifert bundles, determined by  rational divisors on the base orbifold and constructed via Seifert  $S^1$-bundles \cite{HaeSal}.
We can prove the following:

\begin{theorem}\label{fibration}
Let $\pi:M\rightarrow X$ be a quasi-regular fibration such that $X$ has an orbifold line bundle  $L$ with $c_1(L)\geq 0$ and $\pi^*c_1(L) = 0\in H^2_{dR}(M, \R)$. 
Let $k$ be a positive number such that $c_1(L)^k\neq 0$ and  $c_1(L)^{k+1} = 0$. 
Then $M$ does not admit a $(m-j)$-K\"ahler structure,   for any $1\le j\le k$, where $m$ is the complex dimension of $M$.
\end{theorem}

\begin{proof}
In the notation introduced above, $c_1(L)\geq 0$ means that there is a representative $\omega\in c_1(L)$ which is strongly positive on the orbifold $X$.
Let $\theta$ be a $1$-form  such that $d\theta = \pi^*\omega\in \pi^*c_1(L)$. 
Then $\omega^k\neq 0$ and $\omega^{k+1}=0$,  because $\omega^{k+1}$ is closed and non-negative representing the zeroth cohomology class. 
By assumption, $\omega$ is a strongly positive $(1,1)$-form, of rank $k$, so in particular $\omega^j$ is strongly positive and non-zero, for all $1\le j\le k$.
Then for a hypothetical $(m-j)$-K\"ahler structure $\Omega$, with $1\le j\le k$, we have 
\[ \Omega\wedge \pi^*(\omega^j) = \Omega\wedge(d\theta)^j > 0, \hspace{.2in} \int_M\Omega\wedge (d\theta)^j =0,
\]
a contradiction.
\end{proof}

Basic examples are the compact complex homogeneous spaces. If a compact Lie group $G$ acts transitively by biholomorphisms on a manifold $M$, then $M=G/H$ for a subgroup $H$, which is the stationary subgroup at a point $o\in M$. It is known that $G/H$ admits a K\"ahler metric if and only if $H$ is a centralizer of a torus in $G$, and in this case $G/H$ is called generalized flag manifold. For such spaces it is known that they are rational algebraic Fano varieties and admit K\"ahler-Einstein metrics. The following Theorem for the structure of $G/H$ in general is known:

\begin{theorem}[\cite{Tits,wang54}] \label{thmhomog}
Every compact complex homogeneous space $M=G/H$ admitting a transitive action of a compact  Lie group $G$ by biholomorphisms, is a total space of a holomorphic torus fibration $\pi: G/H\rightarrow G/K$,  where $K$ is a centralizer of a torus with $[K,K]=[H,H]$ and $H/K$ a complex tori. The base space $G/K$ is a generalized flag manifold.
\end{theorem}

The structure of the bundle $G/H\rightarrow G/K$ is determined by its characteristic classes which are elements in $H^2_{dR}(G/K, \mathbb{R})$. When for the induced complex structure on $G/K$, $c_1(G/K)$ is in the cone generated by these classes, the space $G/H$ has $c_1(G/H)=0$, see \cite[Theorem 2]{Grantcharov}.
\Cref{thmhomog} was generalized by D. Guan \cite{Guan02} to the case where $G$ is not necessarily compact, but $G/H$ admits an invariant volume form. 
In this case the Tits fibration is $\pi: G/H\rightarrow G_1/K\times D$, where $D$ is a complex parallelizable space, the fibers are tori, and $G_1/K$ is a generalized flag manifold.   
In this case we have:

\begin{corollary}\label{c_1=0}
Suppose that $G/H$ is a compact complex homogeneous space with vanishing first Chern class and invariant volume form, and assume the fibration $\pi: G/H\rightarrow G_1/K\times D$ has rank $r{=\dim_\C D}$, with $dim(G_1/K)=n$.
Then $G/H$ does not admit a $q$-K\"ahler structure, for any $r\le q<n{+r}$.
\end{corollary}
\begin{proof}
We note that we can use \Cref{fibration}, with $L=-K_{G_1/K}$ the anti-canonical bundle, which is positive, and $c_1(L)^n>0$, while $c_1(L)^{n+1}=0$. 
So $k=n$ and the result follows, as the complex dimension of $G/H$ is $n+r$.
\end{proof}

\begin{remark}
\begin{enumerate}
\item As a particular instance, we recover \cite[Corollary 5.1]{FGV}, stating that if a compact complex homogeneous space $M=G/H$ with invariant volume admits a balanced metric, then its first Chern class $c_1(M)$ does not vanish.

\item Both \Cref{fibration} and \Cref{c_1=0} are still valid in much more general terms as long as there is a bundle which contains a characteristic class of appropriate numerical degree $k$.
A partial case of \Cref{c_1=0} is also the case of even-dimensional compact Lie groups. However, instead of formulating the general case, we consider an example below, which elucidates the process of checking the validity of the conclusions. 
\end{enumerate}
\end{remark}

First we start by providing some details in the case of real dimension eight.

\begin{corollary}
Suppose that $M$ is a compact complex manifold  of complex dimension $4$, with a  quasi-regular fibration $\pi\colon M\rightarrow X$, and an orbifold Line bundle $L$ as in \Cref{fibration}, and one of the following is valid:
\begin{enumerate}[label=\roman*)]
    \item  $X$ is a  compact complex orbisurface  and $c_1(L)^2\neq 0$, or
\item  {$X$ a complex orbifold  of  complex dimension $3$} and $c_1^2(L)\geq 0$, but $c_1^2(L)\neq 0$.
\end{enumerate}
Then $M$ does not admit a 2-K\"ahler structure.
\end{corollary}

\begin{proof}
The first case follows from the fact that $\alpha\in c_1(L)$ has $\alpha^2 =\vol_X$   for a fixed volume form $\vol_X$, so $\alpha^2$ is a decomposable 4-form and $\int_X \alpha^2 \neq 0$. For a hypothetical 2-K\"ahler structure $\Omega$, using integration along the fibers, we have 
$$
0=\int_M\Omega\wedge\pi^*\alpha^2 =  \int_X\pi_*\Omega \wedge\alpha^2 = \int_X f \vol_X,
$$
with $f(p)\coloneqq\int_{\pi^{-1}(p)} \Omega$. For dimensional reasons, $\Omega$ is positive along the fibers, so $f$ is a positive function, giving a contradiction.

In the second case, assuming again the existence of a $2$-K\"ahler structure $\Omega$, since the fiber of $\pi$ {has complex dimension one}, $X$ is 1-K\"ahler, which is K\"ahler. Also the positivity of $c_1^2(L)$  in complex dimension $3$ leads to the positivity of $c_1(L)$, so the result follows from \Cref{fibration}.
\end{proof}

We note that, for a fibration over a complex  curve, the {previous}  method can not be applied.

\begin{example}
We consider $\SU(5)/T^2$ in details. The complexification of the Lie algebra of $ \SU(5)$ is $\mathfrak{sl}(m,\mathbb{C})$, for $m=5$. 
For general $m$, let $L_j\in\operatorname{GL}(m,\C)$ be the matrix with $j$-th diagonal element equal to 1 and all others being 0. 
Following \cite{FultonHarris}, the set of all roots of $\mathfrak{sl}(m,\mathbb{C})$ is $\alpha_{j,l} = e_j - e_l$, $j\neq l$, where $e_j$ are the projections to $(\mathfrak{sl}(m,\C))^*$ of the duals of $L_j$, and $e_1+\dots+e_n=0$.
A set of simple roots is $\pg{\alpha_j\coloneqq \alpha_{j,j+1}}$, which also determines the positive roots $\pg{\alpha_{j,l}, j<l}$.
The dual fundamental weights $\pg{\overline{\alpha_j}}$ are defined by
\begin{equation}\label{fundweights}
2\frac{(\overline{\alpha_j},\alpha_l)}{(\alpha_l,\alpha_l)}=\delta^i_l,
\end{equation}
where $(\cdot , \cdot )$ is the bilinear form arising from the Killing form. Then  one can check directly that the following elements satisfy \eqref{fundweights}
$$
\overline{\alpha_j}\coloneqq e_1 +...+e_j.
$$

It is known that the elements $d \, \overline{\alpha_j},  j=1,2,3,4$, where $d$ is the exterior differential operator on $\mathfrak{sl}(5,\C)$,  generate the integral cohomology ring of $\SU(5)/T^4$ and span the second de Rham cohomology group. 
Then a $T^2$-bundle $\pi\colon\SU(5)/T^2\to\SU(5)/T^4$ is defined by a pair of independent characteristic classes
\begin{equation}\label{eqbeta}
    \beta_1=\sum_j a_j \overline{\alpha_j},\quad\beta_2 = \sum_j c_j\overline{\alpha_j}
\end{equation}
with 4-tuples of integers $(a_1, \ldots ,a_4)$ and $(c_1, \ldots ,c_4)$.
 In particular,  any form $A\,d \beta_1 +C\,d\beta_2$ is exact on $\SU(5)/T^2$ for such bundle structure. We want to check when its rank is maximal.
For this, we use 
\begin{equation*}
    \begin{aligned}
d e_l &= -\sum_{m=1}^5 e_{l,m}\wedge e_{m,l},
\\
d\,\overline{\alpha_j}&=-\sum_{l=1}^j\sum_{m=1}^5e_{l,m}\wedge e_{m,l}
=-d\,\overline{\alpha_{j-1}}-\sum_{m=1}^5e_{j,m}\wedge e_{m,j},
    \end{aligned}
\end{equation*}
where $e_{j,l}$ is the dual of the elementary matrix $E_j^l$, and 
\begin{equation}\label{eqdefbeta}
    \beta\coloneqq A\,\beta_1+C\,\beta_2 = \sum_{j=1}^4 B_j\,\overline{\alpha_j},% = \sum B_je_j
\end{equation} 
where $B_j = A\, a_j+C\, c_j$, $j=1,2,3,4$. Now, denoting $\omega_{l,m}=e_{l,m}\wedge e_{m,l},$ and $\tilde{B}_j=B_1+\dots+B_j$, for $j=1,\dots,4$, $\tilde{B}_0=0$, we have
\begin{equation}\label{dbeta}
\begin{aligned}
d\beta &= \sum_{j=1}^4B_j \, d\,\overline{\alpha_j} \\
&=  B_1\,\omega_{1,2} + \pt{B_1+B_2}\omega_{1,3}+ \pt{B_1+B_2+B_3}\omega_{1,4}+ \pt{B_1+B_2+B_3+B_4}\omega_{1,5}\\
&\quad+B_2\,\omega_{2,3}+\pt{B_2+B_3}\omega_{2,4}+\pt{B_2+B_3+B_4}\omega_{2,5}\\
&\quad+B_3\,\omega_{3,4}+\pt{B_3+B_4}\omega_{3,5}+B_4\,\omega_{4,5}
\\
&= \sum_{j<l}\pt{\tilde{B}_{l-1}-\tilde{B}_{j-1}}\omega_{j,l}.
\end{aligned}
\end{equation}
Note that the complex structure $J$ on $\SU(5)$ acts so that $J\alpha_{j,l}=i\alpha_{j,l}$, for $j>l$, and $J\alpha_{j,l}=-i\alpha_{j,l}$, for $j<l$.
In other words, $-i\,\omega_{j,l}$, $j<l$, are positive $(1,1)$-forms, and $-i\,d\beta$ is definite if and only if the $B_j$ have the same sign and are non-zero.
This leads to the conclusion that, for such $B_j$, $\SU(5)/T^2$ does not admit a $p$-K\"ahler structure for any of the $p=2,3,...,10$. 
Of course, if some of the $B_j$ are zero, but not all of them, we will still have an obstruction to the existence of $p$-K\"ahler structures, but only for high values of $p$.
We note here that the condition on $\beta$ is open, so for an open set of left-invariant complex structures there is an obstruction for the existence of $p$-K\"ahler structures.

We can also check that $\SU(5)/T^2$ does not admit a {\it $(n-2)$-pluriclosed} structure for many choices of $\beta_1$ and $\beta_2$, where $n=11$ is the complex dimension of $\SU(5)/T^2$. 
In fact, let $J$ be the complex structure on $\SU(5)/T^2$ defined as the pullback of the one on $\SU(5)/T^4$, on the horizontal subspaces, and $J\beta_1=\beta_2$ on the vertical directions.
Since $d\beta_j$ are $(1,1)$-forms, for $j=1,2$, we have  
$$-dd^c(\beta_1\wedge J\beta_1) =(d\beta_1)^2 + (dJ\beta_1)^2 = (d\beta_1)^2+(d\beta_2)^2.$$
Now, the computations in \eqref{dbeta}, specialized for $\beta_1,\beta_2$, show that for appropriate choices of $a_j$ and $c_j$, the $(1,1)$-forms $d  \beta_1,d\beta_2$ are positive, of rank at least $4$, because every $d\,\overline{\alpha_j}$ has rank at least $4$.
For such values, $(d\beta_1)^2+(d\beta_2)^2$ is then strongly positive.
Since it is a $dd^c$-exact $(2,2)$-form, it provides an obstruction for the existence of $p$-pluriclosed structures, for $p=n-2$, and for some lower values of $p$ as well, depending on the choice of $\beta_1,\beta_2$.
\end{example}

\begin{remark} 
We remark that these observations fit to \cite[Prop. 5.1]{FGV}, stating that $\SU (5)/T^2$  has an invariant complex structure which admits both balanced and astheno-K\"ahler metrics. 
If the $T^2$-bundle structure of  $\SU (5)/T^2$ is generated $\beta_1,\beta_2$ as above, we will consider the complex structure $J$ on $\SU(5)/T^2$ defined similarly as above, as the extension of the pullback of the complex structure on  $\SU (5)/T^4$, but on the horizontal subspaces $J\beta_1=c_J\beta_2$, with $c_J$ a real positive constant. 
Then, for a  K\"ahler metric  $\omega_K$  on $\SU (5)/T^4$, the Hermitian metric $\beta_1\wedge c_J\beta_2+\pi^*\omega_K$ on $\SU(5)/T^2$ is astheno-K\"ahler if and only if 
\begin{equation}\label{aKahSU5}
    \pt{(d\beta_1)^2+c_J^2(d\beta_2)^2}\wedge\pi^*\omega_K^8=0,
\end{equation}
see \cite[Prop. 3.1]{FGV}.
We will choose $\omega_K=d\pt{\overline{\alpha_1}+\overline{\alpha_2}+\overline{\alpha_3}+\overline{\alpha_4}}$, and we will show how different choices of $\beta_1,\beta_2$, both satisfying \eqref{aKahSU5}, can give a complex structure admitting compatible balanced metrics, or not.
We recall that $\beta_1,\beta_2$ are defined in \eqref{eqbeta}, as linear combinations of the fundamental weights $\overline{\alpha_j}$, with integer coefficients.

For the first example, consider $\beta_1=\overline{\alpha_1}$, $\beta_2=\overline{\alpha_1}-\overline{\alpha_2}$, satisfying \eqref{aKahSU5}, with $c_J^2=7/4$.
Furthermore, one can choose $\beta$ as in \eqref{eqdefbeta} so that $d\beta$ is semi-positive, of rank $7$, thus obstructing the existence of $p$-K\"ahler structures, for $4\le p\le 10$, including balanced metrics.

On the other hand, if $\beta_1=\overline{\alpha_1}-3\overline{\alpha_4}$,  $\beta_2=\overline{\alpha_2}-\overline{\alpha_3}$, one can see that \eqref{aKahSU5} still holds, with $c_J^2=7/5$.
However, now no linear combination $\beta$ of $\beta_1$ and $\beta_2$ has $d\beta$ semi-definite, so the existence of balanced metrics is not obstructed.
Moreover, no such $\beta$ intersects the closure of the K\"ahler cone of $\SU (5)/T^4$, so by \cite[Thm. 5.1]{FGV}, there is a balanced metric on $\SU (5)/T^2$ compatible with this complex structure.
We conclude noting that there are some choices of $\beta$ with $(d\beta)^j$ semi-definite, for some $j>1$, for instance $\beta_1^7$ is negative semi-definite, or $\beta_2^4$ is positive semi-definite.
Ultimately, this gives obstructions to the existence of $p$-K\"ahler structures, for $p=2,3,4,7$.

\end{remark}

\section{Reductive Lie groups} \label{section4}

In this section, we investigate invariant complex structures on even-dimensional real reductive Lie groups, distinguishing between regular and non-regular types, and study the existence of special Hermitian metrics.
More precisely, we will begin by recalling a few relevant types of complex structures in this setting, and we will highlight their relation with root spaces. We will then move on to the study of existence of $p$-K\"ahler and $p$-pluriclosed structures, starting with case of rank 1 groups, then turning to the compact groups (\Cref{propcpt}) and concluding with regular complex structures on non-compact Lie groups (\Cref{propred}).

Let $G_0$ be a real Lie group of even dimension.
$G_0$ is called reductive if $\mathfrak g_0$, the Lie algebra of the connected component of the identity, is reductive, i.e. $\mathfrak g_0$ is a direct sum of an abelian ideal and a semisimple ideal.
In this setting, $\mathfrak g$ will denote the complexification of $\mathfrak g_0$, and $G$ will be a complex Lie group with Lie algebra $\mathfrak g$. 
The precise relation between $G_0$ and $G$ is defined through the respective universal coverings, see \cite[(1.1)]{Snow}.

By \cite{Snow}, every invariant complex structure $J$ on $G_0$ is determined by a choice of a {$n$-dimensional} complex subalgebra $\mathfrak q$ of $\mathfrak g$, such that $\mathfrak g=\mathfrak q+\sigma(\mathfrak q)$, where $\sigma$ is the complex conjugation in $\mathfrak g	$ with respect to $\mathfrak g	_0$.
We will denote by $Q$ the complex subgroup of $G$ corresponding to the complex subalgebra $\mathfrak q$, and we will identify the choice of $J$ with $Q$, or, equivalently, with $\mathfrak q$.

\begin{definition}
    An invariant complex structure {on $G_0$} is called \textit{exceptional}, or equivalently, $Q$ is \textit{exceptional}, if $G_0$ is a complex Lie group  and $Q$ is a reductive subgroup of $G$, which is not normal in $G$.
\end{definition}

By \cite[Theorem and Corollary in (1.6)]{Snow}, if  $G_0$ is non-abelian, connected, reductive and  $Q$ is not exceptional, then $Q$ is contained in a proper parabolic subgroup $P$ of $G$.
In particular, $G_0$ fibers holomorphically and equivariantly over an open $G_0$-orbit in $G/P$.

A special class of invariant complex structures, in which case one can say more about existence, classification, and equivalence, is that of \textit{regular} complex structures.
Let $\mathfrak h$ be a Cartan subalgebra of $\mathfrak g$.
\begin{definition}
An invariant complex structure $\mathfrak q$ is {\it regular} with respect to $\mathfrak h$, or \textit{$\mathfrak h$-regular}, if $[\mathfrak h, \mathfrak q]\subset \mathfrak q$.     
If there exists no Cartan subalgebra $\mathfrak h$ of $\mathfrak g$ such that $\mathfrak q$ is $\mathfrak h$-regular, then $\mathfrak q$ is called \textit{non-regular}.
\end{definition}
At the level of the group $G_0$, a regular complex structure is a left-invariant one, which is also right $H$-invariant for the abelian subgroup $H$ with algebra $\mathfrak h$.
If this is the case, one can choose a set of simple roots $\Delta$ such that the root space $\Sigma$ has positive roots $\Sigma^+$ and negative roots $\Sigma^-$ such that $\tau\Sigma^+=\Sigma^-$, where $\tau$ is the conjugation in $\mathfrak g$ defined by $\mathfrak g_0$.
Moreover, every choice of such a $\Delta$ induces the same complex structure, up to equivalence \cite[(2.2)]{Snow}.
If $\mathfrak q$ is a regular complex structure, then one can write 
\begin{equation*}
\mathfrak q=\pt{\mathfrak h\cap\mathfrak q}\oplus\bigoplus_{\alpha\in\Pi}\mathfrak g_\alpha,
\end{equation*}
with ${\mathfrak h\cap\mathfrak q}$ any complex subspace of $\mathfrak h$ of complex dimension equal to half the complex dimension of $\mathfrak h$, and $\Sigma=\Pi\cup\tau\Pi$, $\Pi$ determined by the Levi decomposition $\mathfrak q=\mathfrak u\oplus\mathfrak r$, with $\mathfrak u=\operatorname{rad}(\mathfrak q)$ and $\mathfrak r$ the semisimple part of $\mathfrak q$.
More precisely, associated to $\mathfrak r$ we have a root space $\Pi_{\mathfrak r}= \Pi_\mathfrak r^+\cup \Pi_\mathfrak r^-$, where $\Pi_\mathfrak r^{\pm}=\Pi_\mathfrak r\cap\Sigma^{\pm}$, and $\mathfrak q$ will have root space $\Pi_{\mathfrak u}=\Sigma^+\setminus\left(\Pi_\mathfrak r^+\cup \tau\Pi_\mathfrak r^-\right)$.
Moreover, a set of simple roots of $\mathfrak r$ $\Delta_r$ should be such that $\Delta_r$ and $\mu\Delta_r$ are \textit{separated} sets in the Dynkin diagram for $\Delta$, where $\mu=-\tau$ is an involution of $\Delta$, and actually, Snow proved that every complex structure is determined by a choice of such a $\Delta_r$ and a choice of ${\mathfrak h\cap\mathfrak q}$.
\\
By Morimoto \cite{Morimoto}, every even-dimensional, real, reductive  Lie group admits a regular complex structure, with $\mu$ the identity. 
\\
Based on the possible $\mu$'s the associated space of simple roots admits, reductive Lie groups can be distinguished in two classes.
The terminology is the same for a reductive Lie group and the associated Lie algebra, so we will give the definitions for the Lie algebras.
A real, reductive Lie algebra $\mathfrak g_0$ is called of \textit{Class I}, in the notation of Snow \cite{Snow}, or of \textit{first category}, in the notation of Sugiura \cite{Sugiura}, if its complexified simple factors are all of inner type, or equivalently if the only possible involution of any set of simple roots is the identity.
Every other real, reductive Lie algebra is called of \textit{Class II}, or of \textit{second category}.
\\
Every invariant complex structure on a reductive Lie group of Class I is
regular \cite[Theorem in (3.1)]{Snow}, and whenever $G_0$ is compact, or the semisimple part of $G$ is of rank $1$, we fall in Class I.
Let us begin considering these two cases.

\smallskip
If the semisimple part of $\mathfrak g$, $\mathfrak g_s$ is of rank $1$, then $\mathfrak g$ has to be the direct product of $\mathfrak{sl}(2,\C)$ and an abelian factor $\C^{2n-3}$, and similarly $\mathfrak g_0=\mathfrak{sl}(2,\R)\times\R^{2n-3}$.

\begin{lemma}\label{sl2product}
    Let $J$ be any complex structure on $\mathfrak g_0=\mathfrak{sl}(2,\R)\times\R^{2n-3}$.
    Then, $(\mathfrak g_0,J)$ is a direct product of $(\mathfrak{sl}(2,\R)\times\R,J_1)$ and $(\R^{2n-4},J_2)$, where $J_1,J_2$ are the restrictions of $J$ to the respective ideals of $\mathfrak g_0$.
\end{lemma}

\begin{proof}
Since $\mathfrak g_0=\mathfrak{sl}(2,\R)\times\R^{2n-3}$ is of Class I, every complex structure $J$ on $\mathfrak g_0$ is regular.
Let $\mathfrak h$ be the Cartan subalgebra of $\mathfrak g=\mathfrak g_0^\C$ such that $J$ is $\mathfrak h$-regular.
It follows from the description above that complex conjugation in $\mathfrak g$ with respect to $\mathfrak g_0$ sends positive root spaces to negative root spaces, so the subspace $\mathfrak{sl}(2,\R)\cap J\pt{\mathfrak{sl}(2,\R)}$ of $\mathfrak g_0$ is non-trivial.
Thus, $\mathfrak g_J\coloneqq \mathfrak{sl}(2,\R)+J\pt{\mathfrak{sl}(2,\R)}\subset\mathfrak g_0$ has real dimension $4$, and $(\mathfrak g_J,J)$ is isomorphic to $(\mathfrak{sl}(2,\R)\times\R,J)$, concluding the proof.
\end{proof}

\begin{proposition}\label{propRk1}
    For every $n\ge 2$, $1\le p\le n-1$, and for every complex structure $J$ on $\mathfrak g_0=\mathfrak{sl}(2,\R)\times\R^{2n-3}$, the pair $(\mathfrak g_0,J)$ is $p$-pluriclosed, but not  $p$-K\"ahler.
\end{proposition}

\begin{proof}
First note that $\mathfrak{sl}(2,\R)\times\R$ is Gauduchon, because it is unimodular, but it is not K\"ahler.
This proves the case $n=2$.

For $n\ge2$ we argue by induction.
As a consequence of \Cref{sl2product}, $\mathfrak g_0$ admits a $J$-invariant, central ideal $\mathfrak a$, of real dimension $2$, where $\mathfrak a$ is any $J$-invariant subspace of $(\R^{2n-4},J)$.
This induces a holomorphic submersion of $(\mathfrak g_0,J)$ onto the quotient $\tilde{\mathfrak g}=\mathfrak g_0/\mathfrak a\simeq\mathfrak g_J\times\R^{2n-6}$.
Now, we can use the induction assumption on $\tilde{\mathfrak g}$, together with \cite[Prop. 4.3]{Ales}, proving the statement, for $2\le p<n$.
The case $p=1$ holds true as $(\mathfrak g_0,J)$ is clearly non-K\"ahler, and it is pluriclosed because product of a pluriclosed Lie algebra by a K\"ahler one.
\end{proof}

\begin{remark}\label{rmkbalanced}
    In \cite{GiustiPodesta}, it was proved that every even-dimensional, non-compact, real, simple Lie group $G_0$ of inner type, endowed with an invariant complex structure, admits a balanced metric.
    A straightforward consequence of \Cref{propRk1} is that this cannot be generalized to even-dimensional, non-compact, real, \textit{reductive} Lie group, not even Class I, which is the most natural generalization of inner type simple groups, in the reductive case.
\end{remark}

\smallskip
We will now consider the compact case.
We firstly recall that by \Cref{fibration,thmhomog}, provided the cohomological assumptions in the former, a compact, reductive Lie group of even real dimension cannot be $(n-k)$-K\"ahler, for any $1\le k\le r_0$, where $r_0$ is the complex dimension of the flag manifold, base of the Tits fibrations. In other words, $r_0$ is the number of positive roots of the associated complexified Lie algebra.
We will now show that, using the general structure equations for reductive Lie groups in \cite{AlekPer}, we can prove the same, removing the cohomological assumptions.

\begin{theorem}\label{propcpt}
    Let $G_0$ be a compact, non-abelian,  real reductive Lie group of even dimension $2n$, endowed with an invariant complex structure.
    In the same notation as above, let $r_0$ be the number of positive roots of the complexified Lie algebra $\mathfrak g$.
    Then, $G_0$ cannot be $(n-k)$-K\"ahler, for any $1\le k\le r_0$.
\end{theorem}

\begin{proof}
Having proved \Cref{propRk1}, we can assume $\operatorname{rank}(\mathfrak g_s)\geq 2$.
If $G_0$ is compact, then it is of Class I, thus every invariant complex structure is regular.
Moreover, we can fix a basis $\pg{E_\alpha}_{\alpha\in\Sigma}$ of root vectors such that $\tau E_\alpha=- E_{-\alpha}$.
By \cite[(2.1)]{AlekPer}, if $\xi$ is the highest weight of $\mathfrak g$, then 
$$
d\xi=\sum_{\beta\in\Sigma^+}\xi(H_\beta)\,\omega^\beta\wedge \omega^{-\beta}
$$
where $H_\beta=[E_\beta,E_{-\beta}]$ and $\pg{\omega^\alpha }_{\alpha\in\Sigma}$ is a dual basis of $\pg{E_\alpha}_{\alpha\in\Sigma}$.
Since $\xi$ is the highest weight, $\xi H_\alpha\ge0$, for all $\alpha\in\Sigma^+$, so that $d\xi$ is positive, as $\omega^{-\beta}=\overline{\omega^\beta}$.
Moreover, $d\xi$ has rank equal to the number of positive roots, $r_0=\abs{\Sigma^+}\geq \operatorname{rank}(\mathfrak g)\geq 2$.
As a consequence, for $1\leq k\leq r_0$, $d\xi$ is a non-zero, strongly positive, exact form, obstructing the existence of $(n-k)$-K\"ahler forms on $\mathfrak g_0$.
\end{proof}

For regular complex structures in the non-compact case, we can prove the following.

\begin{theorem}\label{propred}
Let $G_0$ be a non-abelian, real reductive Lie group of even dimension $2n$, with $\operatorname{rank}(\mathfrak g_s)\geq 2$, endowed with an invariant regular complex structure $J$.
Then, $(G_0,J)$ cannot be $(n-2)$-pluriclosed.
\end{theorem}

\begin{proof}
In the same notation as above, we can define the complex constants $\varepsilon_\alpha$, $\alpha\in\Sigma$, such that $J E_\alpha=\varepsilon_\alpha E_\alpha$.
More precisely,
\begin{equation*}
    \varepsilon_\alpha=\begin{cases}
        i,   &   \alpha\in\Pi,    \\
        -i,  &   \alpha\in\tau\Pi,    \\
    \end{cases}
\end{equation*}
so that $\varepsilon_\alpha\varepsilon_{-\alpha}=1$, for all $\alpha\in\Sigma$.
Let $h$ be a symmetric, $J$-invariant, bilinear form on $\mathfrak g$, such that $h(E_\alpha,\cdot)=0$, for all $\alpha\in\Sigma$. 
We can then define a $(1,1)$-form $\omega$, with $\omega(\cdot,\cdot)=h(J\cdot,\cdot)$, and following the proofs of \cite[Lemma 2.2, Lemma 3.1]{LauMon} and \cite[Section 4]{GiustiPodesta}, one can find 
% computation of $dd^c\omega$, where $\omega$ is the $(1,1)$-form associated to a Hermitian metric $h$ on $\mathfrak g$, such that different root spaces are $h$-orthogonal, and such that root 
that the only non-zero components of $dd^c\omega$ are
\begin{equation}\label{eqddc}
    dd^c\omega\pt{E_\alpha,E_{-\alpha},E_\beta,E_{-\beta}}=-2\,h(H_\alpha,H_\beta),
\end{equation}
for all $\alpha,\beta\in\Sigma$, with $\alpha+\beta\neq0$.
This is because, in the references cited above, the same computation is performed when $\omega$ is the $(1,1)$-form associated to some Hermitian metric $g$ on $\mathfrak g$, but the only properties used in the proofs are that $g(E_\alpha,E_\beta)=0$ if $\alpha+\beta\neq0$, $g(\mathfrak h,E_\alpha)=0$, for all $\alpha\in\Sigma$, and that $J E_\alpha=\tilde\varepsilon_\alpha E_\alpha$, with 
\begin{equation*}
    \tilde\varepsilon_\alpha=\begin{cases}
        i,   &   \alpha\in\Sigma^+,    \\
        -i,  &   \alpha\in\Sigma^-.    \\
    \end{cases}
\end{equation*}
However, with our choice of $h$, we can see that the components depending on the $\varepsilon_\alpha$ all vanish, because $h(E_\alpha,\cdot)=0$, allowing to recover \eqref{eqddc}.

As a consequence, if $\operatorname{rank}(\mathfrak g_{s})\ge2$, we find an obstruction to the existence of $(n-2)$-pluriclosed metrics, choosing for instance $h\pt{H_{\alpha_0},H_{\beta_0} }=h\pt{JH_{\alpha_0},JH_{\beta_0} }=1 $, for some $\alpha_0,\beta_0\in\Sigma$ with $\alpha_0+\beta_0\neq0$, and $h(H_\gamma,\cdot)=0$, for $\gamma\in\Sigma\setminus\pg{\alpha_0,\beta_0}$ and $H_\gamma\neq JH_{\alpha_0},JH_{\beta_0}$.
\end{proof}

\begin{remark}
We note that, whenever $\operatorname{rank}(\mathfrak g_s)\geq 4$, the bilinear form $h$ in the proof can be chosen to have more non-degenerate directions, and thus used to obstruct existence of $p$-pluriclosed structures, for lower, even values of $p$.
\end{remark}

\section{Complex structures on \texorpdfstring{$\mathfrak{sl}(2m-1,\R)$}{sl(2m-1,R)}}\label{section5}

This section is dedicated to the construction of non-regular complex structures on $\mathfrak{sl}(2m-1,\R)$, for all $m\ge2$.
Before getting into the details of this construction, let us recall some standard notations and properties of the structure of semisimple Lie algebras, for which we refer to \cite{Knapp}.
Let $\mathfrak g_0$ be a real semisimple Lie algebra, with a Cartan decomposition $\mathfrak g_0=\mathfrak k_0\oplus\mathfrak p_0$ and associated Cartan involution $\theta=\operatorname{id}_{\mathfrak k_0}\oplus\pt{-\operatorname{id}_{\mathfrak p_0}}$.
Let $\mathfrak h_0=\mathfrak t_0\oplus\mathfrak a_0$ be a maximally compact, $\theta$-stable Cartan subalgebra of $\mathfrak g_0$, with $\mathfrak t_0\subset\mathfrak k_0$ and $\mathfrak a_0\subset\mathfrak p_0$.
If $\mathfrak g\coloneqq\mathfrak g_0^\C$ is the complexification of $\mathfrak g_0$, then $\mathfrak h=\mathfrak h_0^\C$ is a Cartan subalgebra of $\mathfrak g$, and $\mathfrak u_0=\mathfrak k_0\oplus i\mathfrak p_0$ is a compact real form of $\mathfrak g$. 
We will denote with $\sigma$ the conjugation in $\mathfrak g$ with respect to $\mathfrak g_0$, and with $\tau$ conjugation with respect to $\mathfrak u_0$. 
A simple computation shows 
\begin{equation}\label{conjInvolution}
	\tau=\sigma\theta=\theta\sigma.
\end{equation}
In what follows, we will always denote with $\Sigma=\Sigma^+\cup\Sigma^-$ a root system for $\mathfrak h$, with positive roots $\Sigma^+$, negative roots $\Sigma^-$ and simple roots $\Delta$.
We note that there cannot be real roots, and that the lexicographic ordering defining $\Sigma^+$ can be chosen with respect to an ordered basis $\mathcal B=\pt{\mathcal B\cap i\mathfrak t_0}\cup \pt{\mathcal B\cap \mathfrak a_0}$, where ${\mathcal B\cap i\mathfrak t_0}$ is a basis of $i\mathfrak t_0$, and $\mathcal B\cap \mathfrak a_0$ is a basis of $\mathfrak a_0$.
This will ensure on the one hand that $\rest{-\tau}{\Sigma^+}=\operatorname{id}_{\Sigma^+}$, and on the other hand that $\theta\Sigma^+=\Sigma^+$, thus $\theta$ is a permutation of the simple roots, with only fixed points the imaginary roots.
The \textit{Vogan diagram} of the triple $\pt{\mathfrak g_0,\mathfrak h_0,\Sigma^+}$ encodes this permutation of $\Delta$, and determines $\mathfrak g_0$, up to isomorphism.

\subsection{Non-regular complex structures on \texorpdfstring{$\mathfrak{sl}(2m-1,\R)$}{sl(2m-1,R)}}

We will now focus on the case $\mathfrak g_0=\mathfrak{sl}(2m-1,\R)$, where $\mathfrak g=\mathfrak g_0^\C=\mathfrak{sl}(2m-1,\C)$.
After specializing the above discussion to this setting, we will construct a new example of a complex structure on $\mathfrak g_0$ (\Cref{lemmaCpxStr}) and ultimately prove that this complex structure is non-regular (\Cref{thmnonreg}).

\smallskip
The ordering of the root system can be chosen so that, given simple roots $\Delta=\pg{\alpha_1,\dots,\alpha_{2m-2}}$, the positive roots can be described as 
\begin{equation*}
	\Sigma^+\coloneqq\pg{\alpha_j^k=\alpha_j+\alpha_{j+1}+\dots+\alpha_{j+k-1},\,1\le k\le 2m-2,\, 1\le j\le 2m-1-k}.
\end{equation*}
In other words, the root $\alpha_j^k$ is the sum of $k$ consecutive simple roots, starting from $\alpha_j$.
For such an ordering of roots, the Cartan matrix is the standard one, i.e.
\begin{equation*}
	\begin{pmatrix}
		2 & -1 & \\
		-1 & \ddots & \ddots \\
		&\ddots & \ddots & -1\\
		&&-1&2
	\end{pmatrix}.
\end{equation*}

Looking at the Vogan diagram of $\mathfrak{sl}(2m-1,\R)$, cf. \cite[pag. 358]{Knapp}, we see that we can assume the Cartan involution to be $\theta(\alpha_j)=\alpha_{2m-1-j}$, and no simple root is purely imaginary, because $\theta$ has no fixed points in $\Delta$.
Furthermore, by the discussion above it follows that $\sigma\pt{\alpha_j}=-\alpha_{2m-1-j}$, and by linearity $\sigma\pt{\alpha_j^k}=-\alpha_{2m-k-j}^k$.
Among positive roots, the only ones preserved by $-\sigma$ are 
\begin{equation*}
	\gamma_j=\alpha_j^{2(m-j)},
\end{equation*}
$1\le j\le m-1$.

We denote by  $\mathfrak g_\alpha$ the $\alpha$--eigenspace for $\operatorname{ad}(\mathfrak h)$, if $\alpha\in\Sigma$, and set $\mathfrak g_\beta=\pg0$, for $\beta\in\mathfrak h^*\setminus\Sigma$.
We can choose generators $e_\alpha$ of $\mathfrak g_\alpha$, $\alpha\in\Sigma^+$, so that 	$\pq{e_\alpha,e_{\beta}}=e_{\alpha+\beta}$, for all $\alpha,\beta\in\Sigma^+$ such that $\alpha+\beta\in\Sigma^+$, with $\alpha<\beta$ in lexicographic order.
%We will also use the notation $$e_j^k=e_{\alpha_j^k}$$
Let $B$ be the Killing form on $\mathfrak g$, and for every $\alpha\in\Sigma$, let $H_\alpha\in\mathfrak h$ be the $B$-dual of $\alpha$, namely the vector defined by $B(H,H_\alpha)=\alpha(H)$, for all $H\in\mathfrak h$.

Then, for all $\alpha_j^k\in\Sigma^+$, $\sigma(\ep{2m-k-j}k)
\in\mathfrak g_{\sigma\alpha_{2m-k-j}^k}=\mathfrak g _{-\alpha_j^k}$ an

d
\begin{equation*}
\pq{\ep jk,\sigma\pt{\ep{2m-k-j}k}}=B\pt{\ep jk,\sigma\pt{\ep{2m-k-j}k}}H_{\alpha	_j^k}.
\end{equation*}
%and $H_{\sigma\alpha}=\sigma(H_\alpha)$, because $\sigma$ is  complex conjugation in $\mathfrak h$.
%Furthermore, if $H_j=H_{\alpha_j}\coloneqq\pq{e_{\alpha_j},e_{-\alpha_j}}$, for $1\le j\le 2(m-1)$, then $\sigma(H_j)=-H_{n-j}$.
For the sake of simplicity, we will denote 
\begin{equation}\label{defhjk}
\begin{aligned}
		&H_j^k=H_{\alpha_j^k},\quad\quad&& B_j^k=B_{\alpha_j^k}=B\pt{e_{\alpha_j^k},\sigma\pt{\ep{2m-k-j}k}},\\
		&H_j=H_j^1,&& B_j=B_j^1,
\end{aligned}\end{equation}
for $\alpha_j^k\in\Sigma^+$.
Note that $\sigma\pt{B_j^k}=B_{2m-k-j}^k$, so in particular $B_j^{2(m-j)}=B_{\gamma_j}\in\R$.
Furthermore, a simple computation shows that, for all $\alpha,\beta,\alpha+\beta\in\Sigma^+$,
\begin{equation}\label{killingsum}
	B_{\alpha+\beta}=B_\alpha B_\beta B\pt{H_\alpha,H_\beta}.
\end{equation}

\bigskip
To prove the existence of a non-regular complex structure, we will generalize complex structure II in \cite{sasaki} to higher dimension.
Before getting into the details of the proof, we explain the idea behind the construction and fix some notation.
Let 
\begin{equation*}
\begin{aligned}
		& \tilde H_k=H_k-2{H_{2m-1-k}}=H_k+\sigma\pt{H_{k}}
		,	&& 		1\le k \le m-3,	\\
		& \tilde H_{m-2}=2\,H_{m-2}+H_{m-1},	\\
		& \tilde H_{m-1}=H_{m-1}+2\,H_{m}=H_{m-1}-2\,\sigma(H_{m-1}).
\end{aligned}
\end{equation*} 
As one can see from the Cartan matrix, $\alpha_{m-1}\pt{\tilde H_k}=0$, for all $1\le k\le m-1$.
Equivalently, denoting  $\tilde{\mathcal H}=\pg{\tilde H_k,1\le k\le m-1}$ and $\tilde{\mathfrak h}=\operatorname{span}_\C{\tilde{\mathcal H}}$,  $\rest{\alpha_{m-1}}{\tilde{\mathfrak h}}=0$.
Moreover,  $\tilde{\mathfrak h}+\sigma{\tilde{\mathfrak h}}=\mathfrak h$, so the complex subalgebra $\mathfrak m$ of $\mathfrak g$ generated by $\tilde{\mathcal H}$ and by all the positive root spaces is a $\mathfrak h$-regular complex structure on $\mathfrak g_0$, see \cite{Morimoto}.
The idea is to start from this complex structure, and to construct the new one replacing $\ep{m-1}{}$ with  $e_0\coloneqq\ep{m-1}{}+\sigma(\ep{m}{})$. 
We will use the notation
$$\hat\Sigma^+\coloneqq\Sigma^+\setminus\pg{\alpha_{m-1}}.$$
This  defines the new complex structure on $\mathfrak g_0$, as explained in the following Lemma.

\begin{lemma}\label{lemmaCpxStr}
	The complex subspace $\mathfrak q$ of $\mathfrak g =\mathfrak{sl}(2m-1,\C)$ generated by
\begin{equation*}
	\mathcal B=
	%\pg{\tilde H_k,\,1\le k\le m-1}
	\tilde{\mathcal H}
	\cup
	\pg{e_\alpha,\,\alpha\in\hat\Sigma^+}	\cup
	\pg{e_0}
\end{equation*}
is a complex subalgebra of $\mathfrak g$, such that $\mathfrak g=\mathfrak q\oplus\sigma(\mathfrak q)$.
In other words, $\mathfrak q$ defines a complex structure $J_{\mathfrak q}$  on $\mathfrak g_0 = \mathfrak{sl}(2m-1,\R)$.
\end{lemma}

\begin{proof}
	The last part of the statement follows from the construction of ${\mathcal B}$, so we only need to show that $\mathfrak q$ is closed under brackets.
	We start noting that, for all $\alpha,\alpha'\in\hat\Sigma^+$, $\pq{e_\alpha,e_{\alpha'}}\subset\mathfrak g_{\alpha+\alpha'}$, where $\mathfrak g_{\alpha+\alpha'}$ is intended to be the null space if $\alpha+\alpha'$ is not a root.
Being $\alpha_{m-1}$ a simple root, this proves that the vector space spanned by $	\pg{e_\alpha,\,\alpha\in\hat\Sigma^+}$ is a subalgebra.
It follows the subspace $\tilde{\mathfrak q}$ of $\mathfrak q$ spanned by $\mathcal B\setminus\pg{e_0}$ is a subalgebra, because $e_\alpha$ is an $\operatorname{ad}(\mathfrak h)$-eigenvector, for all $\alpha\in\Sigma$. 
It remains to prove that $\operatorname{ad}(e_0)$ preserves  ${\mathfrak q}$.
For all $H\in\mathfrak h$, we have $[H,e_0]=\alpha_{m-1}(H)(\ep{m-1}{}-\sigma\pt{\ep m{}})$. 
As noted above, $\alpha_{m-1}$ is null on $\tilde{\mathcal H}$, so $\operatorname{ad}(e_0)$ vanishes on $\tilde{\mathcal H}$.
Moreover, for $\alpha\in\hat\Sigma^+$, $\pq{e_0,e_\alpha}\in \mathfrak g_{\alpha+\alpha_{m-1}}+\mathfrak g_{\alpha-\alpha_{m-1}}$.
Now, since $\alpha$ is a positive root and $\alpha_{m-1}$ is a simple root, $\pq{e_0,e_\alpha}$ is in the space spanned by the positive root vectors.
Furthermore, $\alpha+\alpha_{m-1}$ cannot be $\alpha_{m-1}$ and $\alpha-\alpha_{m-1}\neq\alpha_{m-1}$, for otherwise $2\alpha_{m-1}$ would be a root. 
This yields that $\pq{e_0,e_\alpha}\in\operatorname{span}_\C\pg{e_\alpha,\,\alpha\in\hat\Sigma^+}$, thus concluding the proof.
\end{proof}

We are now ready to prove the following.

\begin{theorem}\label{thmnonreg}
	The real Lie algebra $\mathfrak{sl}(2m-1,\R)$ admits a non-regular complex structure, for all $m\ge2$.
\end{theorem}

\begin{proof}
We will show that the complex structure $\mathfrak q$ constructed in \Cref{lemmaCpxStr} is non-regular, namely that $\mathfrak q$ is not $\mathfrak h_r$-regular, for any $\mathfrak h_r$ Cartan subalgebra of $\mathfrak g$.
We recall that being $\mathfrak h_r$-regular means that $\mathfrak q$ is $\operatorname{ad}(\mathfrak h_r)$-invariant and $\mathfrak h_r=\pt{\mathfrak h_r\cap\mathfrak q}+\sigma\pt{\mathfrak h_r\cap\mathfrak q}$. 
First note that $\mathfrak q$ is not $\mathfrak h$-regular, as $H_{\alpha_{m-1}}\in\mathfrak h$, but $[H_{\alpha_{m-1}},e_0]=\ep{m-1}{}-\sigma\pt{\ep m{}}\notin\mathfrak q$.

We will show that, if some $W\in\mathfrak q$ is such that $\pq{\sigma (W),\mathfrak q}\subset\mathfrak q$, then $W$ is in $\mathfrak h$. %, in contradiction with what was just noted.
Let us fix such a $W=W_\mathfrak h+W_\Sigma+W_0$, where
\begin{equation}\label{defW}
	W_\mathfrak h =\sum_{k=1}^{m-1}c_k\tilde H_k\in\tilde{\mathfrak h} ,\quad 
	W_{\Sigma}=\sum_{\alpha\in\hat\Sigma^+}c_\alpha\ep{}{},
	\quad W_0=c_0\,e_0,
\end{equation}
for some $c_k,c_\alpha\in\C$, $0\le k\le m-1,\alpha\in\hat\Sigma^+$.
The property $[\sigma(W), \mathfrak q]\subset \mathfrak q$ is  equivalent to $[W, \sigma(\mathfrak q)]\subset \sigma(\mathfrak q)$.
Now, $\alpha_m\in\hat\Sigma^+$, so $\ep m{}\in\mathfrak q$, and %$\sigma\pt{\ep m{}}=\sigma\pt{\ep m{}}\in\sigma(\mathfrak q)$. 
%Thus
\begin{equation*}
\begin{aligned}	
	\sigma\pt{\mathfrak q}\ni\pq{W,\sigma\pt{\ep m{}}}%=\pq{W,\sigma\pt{\ep m{}}}
	&=\pq{W_\mathfrak h	,\sigma\pt{\ep m{}}}+\sum_{\alpha\in\hat\Sigma^+}c_\alpha\pq{{\ep{}{}},\sigma\pt{\ep m{}}}+c_0\pq{e_0,\sigma\pt{\ep m{}}}\\
%	&=\sum_{\substack{\alpha\in\hat\Sigma^+\\\alpha-\alpha_{m-1}\in\hat\Sigma^+}}c_\alpha e_{\alpha-\alpha_{m-1}} +c_0\pq{\ep{m-1}{}+\sigma\pt{\ep m{}},\sigma\pt{\ep m{}}}\\
	&=\sum_{\alpha,\alpha-\alpha_{m-1}\in\hat\Sigma^+}c_\alpha\,\varepsilon_\alpha\, e_{\alpha-\alpha_{m-1}} +c_0\,B_{m-1}H_{m-1},
\end{aligned}
\end{equation*}
with $\varepsilon_\alpha\in\C$, where we used that $\pq{W_\mathfrak h,\sigma\pt{\ep m{}}}=-\alpha_{m-1}\pt{W_\mathfrak h	}\sigma\pt{\ep m{}}= 0$.
We recall that $\sigma\pt{H_{{m-1}}}=\sigma\pt{H_{\alpha_{m-1}}}=H_{-\alpha_{m}}=-H_{\alpha_{m}}\notin\mathfrak q$, and $B_{m-1}\neq0$, so  $c_0=0$.
%Similarly, looking at the other summands, we deduce that $c_\alpha=0$, for all $\alpha\in\hat\Sigma^+$ such that $\alpha-\alpha_{m-1}\in\Sigma^+$.

We can use the same argument considering $\ep {2m-1-j}{}\in\mathfrak q$, for all $j\neq m$, to get
\begin{equation}\label{eqsigmaq}
\begin{aligned}	
	%\sigma\pt{\mathfrak q}\ni
	\adws{\ep{2m-1-j}{}}&
	= -\alpha_j\pt{W_\mathfrak h	}\sigma\pt{\ep{2m-1-j}{}}+\sum_{\alpha\in\hat\Sigma^+}c_\alpha \pq{\ep{}{},\sigma\pt{\ep{2m-1-j}{}}}\\
	&=-\alpha_j\pt{W_\mathfrak h	}\emp j{}+\sum_{\alpha\in\Sigma_j}c_\alpha\,\epsilon_\alpha \,e_{\alpha-\alpha_j}+c_{\alpha_j}\,B_{j}H_j,
\end{aligned}
\end{equation}
where  $\epsilon_\alpha\in\C$ and $\Sigma_j=\pg{\alpha\in\hat\Sigma^+\text{ s.t. }\alpha-\alpha_j \in\hat\Sigma^+}\subset \hat\Sigma^+\setminus\pg{\alpha_j}$.
Focusing on the last summand, we can argue, as above, that $H_{j}$ is not in $\sigma\pt{\mathfrak q}$, so $c_{\alpha_j}=0$, for all $1\le j\le 2m-2$ with $j\neq m$.
As for the second sum, recall that if $\alpha-\alpha_j\in\Sigma$, then $\alpha-\alpha_j$ is actually a positive root, because $\alpha$ is positive, and $\alpha_j$ simple.
It follows that the second summand cannot be in $\sigma\pt{\mathfrak q}$, unless zero, because $\mathfrak q$ does not contain negative roots.
Thus, for every $1\le j\le 2m-2$ with $j\neq m$, for every $\alpha\in\hat\Sigma^+$ such that $\alpha-\alpha_j\in\Sigma$, we must have $c_\alpha=0$.
We claim that, since $j$ is arbitrary, this implies $c_{\alpha_l^k}=0$, for all $2\le k\le 2m-2$, $1\le l\le 2m-1-k$, ultimately proving that $c_\alpha=0$, for all $\alpha\in\hat\Sigma^+$, as 
\begin{equation*}
	\hat\Sigma^+=\pg{\alpha_j,\,1\le j\le 2m-2,\,j\neq m}\cup\pg{\alpha_l^k,\,2\le k\le 2m-2,\,1\le l\le 2m-1-k}.
\end{equation*}
To prove the claim, we recall that $\alpha_l^k=\alpha_l+\dots+\alpha_{l+k-1}$, and we distinguish in two cases, the first one being when $l\neq m$, and the second one $l=m$.
In the first case, we pick $j=l$ and we note that since $k\ge2$, $\alpha_l^k-\alpha_l=\alpha_{l+1}+\dots+\alpha_{l+k-1}\in\Sigma^+$. 
On the other hand, if $l=m$, then $l+k-1=m-1+k> m$, where we used once again $k\ge2$. Thus, we can choose $j=m-1+k$ to get $\alpha_m^k-\alpha_{m-1+k}=\alpha_{m}+\dots+\alpha_{m+k-2}\in\Sigma^+$, proving the claim. 

To sum up, we showed that for every $W=W_\mathfrak h+W_\Sigma+W_0\in\mathfrak h\cap\mathfrak q$, defined as in \eqref{defW}, such that $\mathfrak q$ is $\operatorname{ad}_{\sigma(W)}$-invariant, then $W_0=W_\Sigma=0$, or equivalently $W\in\mathfrak h$. 
In other words, if $\mathfrak q$ is regular, with respect to some Cartan subalgebra $\mathfrak h_r$, then $\mathfrak h_r\subseteq\mathfrak h$ and so these two spaces are the same.
This is a contradiction, as we already noted that $\mathfrak q$ is not $\mathfrak h$-regular.
\end{proof}

We will now discuss the existence of special Hermitian metrics for this class of complex structures.

\begin{theorem}\label{thmSL}
The complex structure $J_{\mathfrak q}$ on $\mathfrak{sl}(2m-1,\mathbb{R})$, as defined in \Cref{lemmaCpxStr}, admits compatible balanced metrics.
	\end{theorem}

\begin{proof}
	We recall that by \cite[Lemma 1]{AV}, a Hermitian metric on a $2n$-dimensional  unimodular real  Lie algebra $(\mathfrak{g}_0, J)$ endowed with a complex structure $J$ is balanced if and only if there is a basis $\pg{v_1,\dots,v_n}$ of $\mathfrak g_0^{1,0}$ such that 
\begin{equation}\label{balancedn}
    \sum_{j=1}^n\pq{v_j,\overline{v_j}}=0,
\end{equation}
and in our particular case, complex conjugation is $\sigma$, so to prove the existence of a balanced metric compatible with $J_{\mathfrak q}$ we need to find $X_1,\dots,X_n$ generating $\mathfrak q$ satisfying 
\begin{equation*}
    \sum_{j=1}^n\pq{X_j,\sigma\pt{X_j}}=0,
\end{equation*}
with $n=2m(m-1)=\dim_\C\mathfrak q$.
We start noting that the only $X\in\mathcal B$ with $\pq{X,\sigma(X)}\neq0$ are $e_0$ and $e_{\gamma_j}$, where we recall that $\gamma_j=\alpha_j^{2(m-j)}$, $j=1,\dots,m-1$, are the only positive roots preserved by $-\sigma$.
%Furthermore, by construction,
%\begin{equation*}
%	\pq{e_{\gamma_j},\sigma\pt{e_{\gamma_j}}}=\pq{e_{\gamma_j},{e_{-\gamma_j}}}=h_j^{2(m-j)}.
%\end{equation*}
We define $f_\alpha$, $\alpha\in\Sigma^+$, as
\begin{equation*}
\begin{aligned}
	&f_{\alpha_{m-1}}\coloneqq e_0-B_{m-1}\, {e_{\alpha_m}} =\ep{m-1}{}+\sigma\pt{e_{\alpha_m}}-B_{m-1}\, {e_{\alpha_m}},	\\
	&f_{\gamma_{m-1}}\coloneqq e_{\gamma_{m-1}} +\frac{\tilde{H}_{m-1}}{3}=e_{\gamma_{m-1}} +\frac{{H}_{m-1}+2H_m}{3},	\\
	&f_{\alpha_j^{m-j}}\coloneqq e_{\alpha_j^{m-j}}-B_j^{m-j}B\pt{H_j^{m-j},H_m^{m-j}}e_{\alpha_m^{m-j}},
		&& 1\le j<m-1,\\
	&f_\alpha=e_\alpha,
		&&\alpha\neq\gamma_{m-1},\alpha_j^{m-j},\,j\le  m-1.
\end{aligned}
\end{equation*}
%$f_\alpha=e_\alpha$, for $\alpha\neq\alpha_{m-1},\gamma_{m-1},\alpha_j^{m-j},1\le j< m-1$.
We note that $\gamma_j=\alpha_j^{m-j}+\alpha_m^{m-j}$, and $B\pt{H_j^{m-j},H_m^{m-j}}\in\R$, for $1\le j\le m-1$. 
This construction ensures that
\begin{equation*}
	\tilde{\mathcal B}=	\tilde{\mathcal H}\,\cup
	\pg{f_\alpha,\,\alpha\in\Sigma^+}
\end{equation*}
is a basis of $\mathfrak q$. 
Now, the only  $X\in\tilde{\mathcal B}$ with $\pq{X,\sigma(X)}\neq0$ are $f_{\gamma_j}$, $f_{\alpha_j^{m-j}}$, $1\le j\le m-1$.
One can easily compute, for $j=m-1$,
\begin{equation*}
\begin{aligned}
		&\pq{f_{\alpha_{m-1}},\sigma\pt{f_{\alpha_{m-1}}}}={e_{\gamma_{m-1}}-\sigma\pt{e_{\gamma_{m-1}}}}-B_{m-1}B_m\,H_{\gamma_{m-1}},\\
		&\pq{f_{\gamma_{m-1}},\sigma\pt{f_{\gamma_{m-1}}}}=B_{\gamma_{m-1}}H_{\gamma_{m-1}}-{e_{\gamma_{m-1}}+\sigma\pt{e_{\gamma_{m-1}}}},
\end{aligned}
\end{equation*}
where we used that $B_m=\sigma\pt{B_{m-1}}$. 
By \eqref{killingsum}, we get
\begin{equation*}
	\pq{f_{\alpha_{m-1}},\sigma\pt{f_{\alpha_{m-1}}}}+\pq{f_{\gamma_{m-1}},\sigma\pt{f_{\gamma_{m-1}}}}=0.
\end{equation*}
Similarly, for $j<m-1$, using that $\sigma\pt{B_j^{m-j}}=B_m^{m-j}$, 
\begin{equation*}
\begin{aligned}
		&\pq{f_{\alpha_j^{m-j}},\sigma\pt{f_{\alpha_j^{m-j}}}}=-B_j^{m-j}B_m^{m-j}B\pt{H_j^{m-j},H_m^{m-j}}\,(H_j^{m-j}+H_m^{m-j}),\\
		&\pq{f_{\gamma_{j}},\sigma\pt{f_{\gamma_{j}}}}=B_{\gamma_{j}}H_{\gamma_{m-1}}.
\end{aligned}
\end{equation*}
Once again, by \eqref{killingsum}, the sum of the two lines vanish, ultimately proving that the Hermitian metric with unitary basis $\tilde{\mathcal B}$ is balanced.
\end{proof}

\begin{example}\label{sl3}
    For $m=3$, we have $\mathfrak g_0=\mathfrak{sl}(3,\R)$, and the construction above gives the complex structure denoted in \cite{sasaki,AGT} by II. In this case, we can write explicitly the complex structure equations, in terms of a basis  $\pg{\alpha^0,\dots,\alpha^4}$ dual to   $\mathcal B$ :
\begin{equation*}
    \begin{cases}
        d\alpha^0=\frac13\pt{\alpha^{1\bar2} -2\alpha^{2\bar1}-\alpha^{3\bar3}},\\
        d\alpha^1=-3\alpha^{1\bar0}-\alpha^{3\bar1},\\
        d\alpha^2=-3\alpha^{02}-\alpha^{13}-6\alpha^{0\bar1}-\alpha^{1\bar3}+\alpha^{3\bar2},\\
        d\alpha^3=-3\alpha^{03}-\alpha^{12}-\alpha^{1\bar1}-3\alpha^{3\bar0}.
    \end{cases}
\end{equation*}
In particular,
\begin{equation*}
    \begin{aligned}
       & \partial\bar\partial\pt{\alpha^{1\bar1}}= 3\alpha^{01\bar0\bar1}+\frac13\alpha^{13\bar1\bar3},\\
       & \partial\bar\partial\pt{\alpha^{0\bar01\bar1}}= 3\alpha^{013\bar0\bar1\bar3},
    \end{aligned}
\end{equation*}
proving that $\pt{\mathfrak{sl}(3,\R),J}$ cannot admit $p$-pluriclosed structures, for $p=1,2$.
As a consequence, it does not admit pluriclosed nor astheno-K\"ahler metrics.
\end{example}

\begin{corollary}\label{slnSKT}
    The pair $(\mathfrak{sl}(2m-1,\R), J_{\mathfrak q})$, where $J_{\mathfrak q}$ is the complex structure from \Cref{lemmaCpxStr}, cannot  {admit  compatible pluriclosed metrics}.
\end{corollary}

\begin{proof}
$(\mathfrak{sl}(2m-1,\R), J_{\mathfrak q})$ contains the complex subalgebra $(\mathfrak{sl}(3,\R),J_{II})$ described in \Cref{sl3}, generated by $\mathcal B_3=\pg{\tilde H_{m-1},e_0,e_{\alpha_m},e_{\gamma_{m-1}}}$ and $\sigma\pt{\mathcal{B}_3}$.
We just showed that this is not pluriclosed, so $(\mathfrak{sl}(2m-1,\R), J_{\mathfrak q})$ cannot be either.
\end{proof}

\begin{remark}
We note that the results in \Cref{thmSL} and \Cref{slnSKT} are in accordance with the Fino-Vezzoni conjecture \cite{FV}, stating that a compact, complex, non-K\"ahler manifold cannot admit both balanced and pluriclosed metrics.
\end{remark}

\subsection{Regular complex structures on  \texorpdfstring{$\mathfrak{sl}(3,\R)$}{sl(3,R)}}

We know by \cite{AGT} that, besides $J$, $\mathfrak{sl}(3,\R)$ admits just one more family of complex structures, up to isomorphism.
This family, $\I$, $\abs{\lambda}<1$, is defined by having the following complex structure equations
\begin{equation*}\begin{aligned}
    &\pq{u,x}=(2-\lambda)x,\quad
    \pq{u,y}=(2\lambda-1)y,\quad
    \pq{u,z}=(\lambda+1)z,\quad
    \pq{x,y}=z,\\
    &\pq{u,\bar x}=(1-2\lambda)\bar x,\quad
    [u,\bar y]=(\lambda-2)\bar y,\quad
    [u,\bar z]=-(\lambda+1)\bar z,\\
    &[x,\bar y]=\frac{u+\lambda\bar u}{1-\abs\lambda^2},\quad
    [x,\bar z]=-\bar x,\quad
    [y,\bar z]=\bar y,\quad
    [z,\bar z]=\frac{(1-\bar\lambda)u-(1-\lambda)\bar u}{1-\abs\lambda^2}.
\end{aligned}\end{equation*}
It is then clear that $\I$ is regular with respect to $\mathfrak h=\langle u,\bar u\rangle\subset\mathfrak{sl}(3,\C)$.
Moreover, with the same argument as above, we can see that the Hermitian metric on $\pt{\mathfrak{sl}(3,\R),\I}$, with unitary basis $\pg{u,x,x-y,z}$, is balanced, proving the following.

\begin{proposition}
    Every complex structure on $\mathfrak{sl}(3,\R)$ admits a {compatible} balanced metric.
\end{proposition}

\bigskip

{\bf Acknowledgements.} 
This collaboration started while third named author was a Postdoctoral Researcher at ICMS and IMI-BAS, which she warmly thanks for the hospitality.

Anna Fino is partially supported by Project PRIN 2022 \lq \lq Geometry and Holomorphic Dynamics”, by GNSAGA (Indam) and by a grant from the Simons Foundation (\#944448).  Gueo Grantcharov is partially supported by a grant from the Simons Foundation (\#853269).
Asia Mainenti is partly supported by the Bulgarian Ministry of Education and Science, Scientific Programme "Enhancing the Research Capacity in Mathematical Sciences (PIKOM)", No. DO1-67/05.05.2022 and by the PNRR-III-C9-2023-I8 grant CF 149/31.07.2023 {\em Conformal Aspects of Geometry and Dynamics}.

\bibliographystyle{plain}

\end{document}